\documentclass{article}

\usepackage{amssymb,amsthm,amsmath}
\usepackage[pdftex]{hyperref}
\newcommand{\Z}{\mathbb{Z}}
\newcommand{\Q}{\mathbb{Q}}
\renewcommand{\O}{\mathcal{O}}

\theoremstyle{definition}
  \newtheorem{definition}{Definition}[section]
  \newtheorem{example}[definition]{Example}

\theoremstyle{theorem}
 \newtheorem{proposition}[definition]{Proposition}
 \newtheorem{lemma}[definition]{Lemma}
 \newtheorem{theorem}[definition]{Theorem}

\title{On Certain Divisibility Property of Polynomials}
\author{Luis F. C\'aceres \& Jos\'e A. V\'elez Marulanda}
\begin{document}
\maketitle
\begin{abstract}
We review the definition of D-rings introduced by H. Gunji \& D. L. MacQuillan. We provide an alternative characterization for such rings that allows us to give an elementary proof of that a ring of algebraic integers is a D-ring. Moreover, we give a characterization for D-rings that are also unique factorization domains to determine divisibility of polynomials using polynomial evaluations.
\end{abstract}
\section{Introduction}
Assume that $f(x)$ and $g(x)$ are polynomials with integer coefficients. Assume also that for all $k\in \Z$ such that $g(k)\not=0$ we have $g(k)|f(k)$ in $\Z$. Can we say that $g(x)|f(x)$ in $\Z[x]$? Certainly this is not always true. For instance, if we consider $p$ any prime, $f(x)=x^p-x$ and $g(x)=p$, using Fermat's little theorem we see that for all $k\in \Z$, $g(k)|f(k)$, but $g(x)\nmid f(x)$ in $\Z[x]$ although $\frac{f(x)}{g(x)}\in \Q[x]$. This motivates the following definition introduced by H. Gunji \& D. L. MacQuillan in \cite{hiroshi}. \begin{definition}
Let $R$ be an integral domain. $R$ is a \textit{D-ring} if given polynomials $f(x)$ and $g(x)$ with coefficients in $R$ such that for almost all $k\in R$, $g(k)|f(k)$ then 
$\frac{f(x)}{g(x)}\in K[x]$,
where $K$ is the field of fractions of $R$.
\end{definition}  
Fields cannot be D-rings: if $F$ is a field, trivially $k|1$ for all nonzero element $k\in F$ but $\frac{1}{x}\not\in F[x]$.  From now on, we always assume all the rings are commutative with identity.

Let $R$ be a ring. For any polynomial $f(x) \in R[x]$ denote $S(f)$ the set of all nonzero prime ideals $\mathfrak{p}$ of $R$ such that the congruence $f(x)\equiv 0\mod\mathfrak{p}$ is solvable in $R$, i.e. there exists $k\in R$ such that $f(k)\in \mathfrak{p}$. In particular, if $c\in R$, $S(c)$ is precisely the set of prime ideals of $R$ that contain $c$.
\begin{proposition}\label{dringeq}
Let $R$ be an integral domain and $R^\times$
the multiplicative group of  $R$ . The following properties are equivalent:
\begin{enumerate}
\item[(i)] $R$ is a D-ring.
\item[(ii)] Every polynomial over $R$ which satisfies $f(k)\in R^\times$ for almost all $k\in R$ must be a constant.
\item[(iii)] For any nonconstant polynomial $f(x)\in R[x]$, the set $S(f)$ is nonempty.
\item[(iv)] For any nonconstant polynomial $f(x) \in R[x]$ and any non-zero $c\in R$, the set $S(f)-S(c)$ is infinite.
\end{enumerate}
\end{proposition}
\begin{proof}
See \cite[Prop 1, pg 290]{hiroshi}.
\end{proof}
As a consequence of Proposition \ref{dringeq}, the ring  $\Z[W]$ where $W=\{1/p: p \text{ is prime and } p\equiv 1\mod 4 \text { or } p=2 \}$ is not a D-ring using the nonconstant polynomial $f(x)=x^2+1$ and the fact that for all $k\in \Z[W]$, $f(k)\in \Z[W]^\times$ (see \cite[Example 1, pg 293]{hiroshi}).  

We review the definition of algebraic integers. Let $R$ be a subring of a ring $L$. An element $\alpha\in L$ is {\it integral} over $R$ if there exists a monic polynomial $f(x)\in R[x]$ such that $f(\alpha)=0$. In particular, when $R=\Z$, the element $\alpha$ is said to be an {\it algebraic integer} in $L$. It is well-known that the set $B$ consisting of all the elements that are integral over $R$ is a ring which is called the {\it integral closure} of $R$ in $L$.  In particular,  if $R=\Z$ and $L$ is a field containing $\Z$, the integral closure of $\Z$ in $L$  is called the {\it ring of integers} of $L$, and we denote this ring by $\O_L$. For example, let $d$ be a square-free integer and consider $\Q(\sqrt{d})=\{a+b\sqrt{d}: a, b\in \Q\}$, then  ring of integers in $\Q(\sqrt{d})$ is $\O_{\Q(\sqrt{d})}=\Z[\omega]=\{a+b\omega: a, b\in \Z\}$ where 
\begin{equation*}\label{w}
\omega=\begin{cases} \sqrt{d}, &\text{ if $d\equiv 2, 3\mod 4$}\\
\frac{1+\sqrt{d}}{2}, &\text{ if $d\equiv 1\mod 4$}\end{cases}
\end{equation*}
We say that an integral domain $R$ is {\it integrally closed} if $R$ is equal to its integral closure in its field of fractions. For example, the integral closure of $\Z$ in $\Q$ is itself, which implies that $\Z$ is integrally closed. 

It is stated in \cite[Cor 1, pg 293]{hiroshi} that the following result is a direct consequence of Proposition \ref{dringeq}.
\begin{proposition}\label{propint}
Let $\Q\subseteq L$ be a finite Galois extension of fields. Then the ring of algebraic integers $\O_L$ in $L$ is a D-ring. 
\end{proposition}
We provide in the following section we prove that $\Z$ is a D-ring (see Lemma \ref{Zisadring}). Then we can see that Proposition \ref{propint} is actually a consequence of the following result. 
\begin{proposition}\label{propint2}
Let $R$ be an integrally closed domain and $K$ be its ring of fractions. Let $K\subseteq L$ be a finite Galois extension of fields and let $C$ be the integral closure of $R$ in $L$. If $R$ is a D-ring then $C$ is also a D-ring. 
\end{proposition}
In the following section we provide an elementary proof of Proposition \ref{propint2} which proves Proposition \ref{propint} without using Proposition \ref{dringeq}. 

\section{Alternative Characterizations of D-rings}

\begin{proposition}\label{dringdpp}
Let $R$ be an integral domain and let $K$ be its field of fractions. $R$ is a D-ring if and only if for given polynomials $f(x)$ and $g(x)$ in $R[x]$ such that for $k\in R$  $(g(k)\not=0\Rightarrow g(k)|f(k))$  then either $f(x)=0$ or $\deg g\leq \deg f$. 
\end{proposition}
\begin{proof}
$(\Rightarrow)$ Assume that $R$ is a D-ring.  Let $g(x),f(x) \in R[x]$ such that for all $k\in R$ with $(g(k)\not=0 \Rightarrow g(k)\vert f(k))$. Since $g(x)$ has finitely many zeros, $g(k)\vert f(k)$ for almost all $k \in R$. Since $R$ is a D-ring, $\frac{f(x)}{g(x)}\in K[x]$.  Then there exists $p(x)\in K[x]$ such that $f(x)=p(x)g(x)$. Suppose $f(x)\not= 0$, then $\deg f = \deg (pg) = \deg p + \deg g\geq \deg g$. $(\Leftarrow)$ Let $g(x),f(x)\in R[x]$ such that for almost all $k\in R$, $g(k)\vert f(k)$. Let $A=\{k_1,\ldots,k_n\}$ be a finite subset of $R$ such that $g(k)\vert f(k)$ for all $k \in R-A$. Let $k_1,\ldots,k_s \in A$ such that $g(k_i)\not= 0$ for $i=1,\ldots,s$ and let $\beta = g(k_1)\cdots g(k_s)$. If $s=0$, let $\beta = 1$. Consequently, for all $k\in R$ $(g(k)\not= 0  \Rightarrow g(k)\vert \beta f(k))$. By hypothesis, $\beta f(x)=0$ or $\deg g \leq \deg \beta f$. If $\beta f(x)=0$, then $f(x)=0$, which trivially implies $\frac{f(x)}{g(x)}\in K[x]$.  Suppose $\deg g \leq \deg \beta f$ and let $g(x)= a_nx^n+\cdots+a_0$. By the division algorithm there exist $q(x),r(x) \in K[x]$ and $s\in \Z^+$ such that $a_n^s\beta f(x)= g(x)q(x)+r(x)$, with $r(x)=0$ or $\deg r < \deg g$. Let $\alpha=a_n^s\beta$ and  suppose that $\deg r<\deg g$. Therefore for all $k\in R$ with $g(k)\not = 0$ we have both $g(k)\vert\alpha f(k)$ and $g(k)\vert g(k)q(k)$, implying $g(k)\vert r(k)$. Using the hypothesis for the polynomials $g(x)$ and $r(x)$  we obtain $r(x) = 0$ or $\deg r \geq \deg g$. Therefore $r(x) = 0$ and hence $\alpha f(x) = g(x)q(x)$. It follows that $\frac{f(x)}{g(x)} = \alpha^{-1}q(x)\in K[x]$. In others words, $R$ is a $D$-ring.
\end{proof}

\begin{lemma}\label{Zisadring}
Let $g(x),f(x)\in \Z[x]$ such that for $k\in \Z$, $(g(k)\not=0\Rightarrow g(k)|f(k))$. Then $f(x)=0$ or $\deg g\leq \deg f$. Consequently, $\Z$ is a D-ring.
\end{lemma}
\begin{proof}
Let $g(x)= a_nx^n+\cdots + a_1x+a_0$ and $f(x)=b_mx^m+\cdots + b_1x+b_0$ be polynomials in $\Z[x]$ such that $f(x)\not=0$ and $\deg g=n>\deg f$. Without loss of generality, assume $a_n,b_m > 0$. Then  we can find $k\in \Z$ large enough such that $g(k)\not=0$ and $a_nk^n+\cdots+a_1k+ a_0 > b_mk^m+\cdots+b_1k+ b_0$ which implies that $g(k)\nmid f(k)$. This proves the first conclusion of Lemma \ref{Zisadring} by contradiction. The second conclusion follows from Proposition \ref{dringdpp}
\end{proof}

\begin{proposition}\label{intelements}
Let $R$ be an integral domain and $K$ be its field of fractions. Assume that $K\subseteq L$ is a finite Galois extension of fields and let $C$ be the integral closure of $R$ in $L$. Then $\sigma(C)=C$ for all $\sigma \in \text{Gal}(L/K)$. Moreover, if $R$ is integrally closed, then $R=\{b\in C:\sigma(b)=b, \text{ for all } \sigma \in \text{Gal}(L/K)\}$.
\end{proposition} 
For a proof of Proposition \ref{intelements} see for example \cite[Prop  2.19]{lorenzini}.

Let $R$, $K$, $L$ and $C$ as in the hypothesis of Proposition \ref{intelements}. Assume that $p(x)=\alpha_nx^n+\cdots+a_1x+a_0\in L[x]$ and  $\sigma\in \text{Gal}(L/K)$ are arbitrary and let  $p_\sigma(x)=\sigma(\alpha_n)x^n+\cdots\sigma(\alpha_1)x+\sigma(\alpha_0)$. Observe that if $p(x)=r(x)s(x)$ with $r(x), s(x) \in L[x]$ then $p_\sigma(x)=r_\sigma(x)s_\sigma(x)$. Moreover, if $a\in R$ then $p_\sigma(a)=\sigma(p(a))$.  It follows from Proposition \ref{intelements} that $p_\tau(x)=p(x)$ for all $\tau\in \text{Gal}(L/K)$ if and only if $p(x)\in R[x]$. Let $N_{L/K}(p)(x)=\prod_{\sigma\in \text{Gal}(L/K)}p_\sigma(x)$. It follows that for all $a\in R$, $N_{L/K}(p)(a)=\prod_{\sigma\in \text{Gal}(L/K)}p_\sigma(a)=\prod_{\sigma\in \text{Gal}(L/K)}\sigma(p(a))$. 
\begin{lemma}\label{polnorm}
Let $R$, $K$, $L$ and $C$ as in the hypothesis of Proposition \ref{intelements} and let $p(x)\in C[x]$ be arbitrary.
\begin{itemize}
\item[(i)] $\deg N_{L/K}(p)=\left| \text{Gal}(L/K)\right| \deg p$.
\item[(ii)] $N_{L/K}(p)(x)=0$ if and only if $p(x)=0$.
\item[(iii)] $N_{L/K}(p)(x)\in R[x]$.
\item[(iv)] $N_{L/K}(p)(a)\in R$ for all $a\in R$. 
\end{itemize}
\end{lemma}
\begin{proof}
Statements (i) and (ii) follow directly from the definition of $N_{L/K}(p)(x)$. Let $q(x)=N_{L/K}(p)(x)$ (note that $q(x)\in C[x]$). Let $\tau\in \text{Gal}(L/K)$ be fixed but arbitrary. Then $\tau\circ \sigma\in\text{Gal}(L/K)$ and  ${(p_{\tau})}_\sigma(x)=p_{\tau\circ \sigma}(x)$ for all $\sigma\in \text{Gal}(L/K)$. Note that $\tau$ induces a permutation of the finite group $\text{Gal}(L/K)$. Then $q_\tau(x)=\prod_{\tau\circ\sigma\in \text{Gal}(L/K)}p_{\tau\circ \sigma}(x)=q(x)$ which implies $q(x)\in R[x]$ proving (iii). Note that (iv) is a direct consequence of (iii).
\end{proof}
\begin{proof}[Proof of Proposition \ref{propint2}]
Let $f(x),g(x)\in C[x]$ be such that for all $k\in C$ $(g(k)\not=0 \Rightarrow g(k)|f(k))$. Consider $F(x)=N_{L/K}(f)(x)$ and $G(x)=N_{L/K}(g)(x)$. By Lemma \ref{polnorm}, $F(x), G(x)\in R[x]$. Let $b\in R$ such that $G(b)\not=0$. Then $g(b)\not=0$ which implies $g(b)|f(b)$. Therefore $\sigma(g(b))|\sigma(f(b))$ for all $\sigma \in \text{Gal}(L/K)$. Using properties of divisibility together with Lemma \ref{polnorm} we obtain  
\[G(a)=\left(\prod_{\sigma\in \text{Gal}(L/K)}\sigma(g(a))\right)\mid\left(\prod_{\sigma\in \text{Gal}(L/K)}\sigma(f(a)))\right)=F(a)\]. 
We have proved that for every $b\in R$, $(G(b)\not=0 \Rightarrow G(b)|F(b))$. Using Proposition \ref{dringdpp}, we obtain that either  $\deg G\leq \deg F$ or $F(x)=0$. Using Lemma \ref{polnorm} we obtain $\deg g\leq \deg f$ or $f(x)=0$. It follows from Proposition \ref{dringdpp} that $C$ is a D-ring.
\end{proof}

It is clear that Proposition \ref{propint} follows directly from Proposition \ref{propint2} together with Lemma \ref{Zisadring}. 

In the following, we give a characterization of D-rings that are also unique factorization domains.

Let $R$ be a unique factorization domain.  For all $p(x)\in R[x]$ we denote by $C(p(x))$ the {\it content} of $p(x)$, i.e. the greatest common divisor of the coefficients of $p(x)$. Remember that $p(x)\in R[x]$ is said to be \textit{primitive} if $C(p(x))$ is a unit. Gauss' lemma states that the product of two primitive polynomials over a unique factorization domain is also primitive. This implies that if $f(x), g(x)$ and $h(x)$ are polynomials in $R[x]$ with $g(x)$ primitive and $mh(x)=f(x)g(x)$ for some $m\in R$, there exists a polynomial $q(x)\in R[x]$ such that $f(x)=mq(x)$. 
\begin{proposition}\label{dringepp}
Let $R$ be a unique factorization domain and let $K$ its field of fractions. $R$ is a D-ring if and only if given $f(x),g(x) \in R[x]$ with $g(x)$ nonconstant and primitive such that
for all $k \in R$, $(g(k)\not=0 \Rightarrow g(k)\vert f(k))$, then $g(x)\vert f(x)$ in $R[x]$. 
\end{proposition}
\begin{proof}
$(\Rightarrow)$. Let $f(x),g(x) \in R[x]$ with $g(x)$ nonconstant primitive such that for all $k \in R$, $g(k)\not=0 \Rightarrow g(k)\vert f(k)$. It is clear that for almost all $k\in R$, $g(k)\vert f(k)$. Since $R$ is a D-ring we have that $\frac{f(x)}{g(x)}=p(x)\in K[x]$. Let $p(x) = \frac{r_n}{s_n}x^n+\frac{r_{n-1}}{s_{n-1}}x^{n-1}+\cdots+\frac{r_1}{s_1}x+\frac{r_0}{s_0}$, where $r_i,s_i \in R$, with $s_i \not= 0$ for all $i = 0,\ldots,n$.
Let $m = s_0s_1\cdots s_n$, therefore  $mp(x)\in R[x]$. Take $h(x)=mp(x)$. We have $mf(x) = mp(x)g(x)=h(x)g(x)$, with $g(x)$ primitive. Then there exists $q(x)\in R[x]$ such that $h(x) = mq(x)$, and so $mf(x) = mq(x)g(x)$.
Therefore $f(x) = q(x)g(x)$, with $q(x) \in R[x]$; i.e. $g(x)\vert f(x)$ in $R[x]$. $(\Leftarrow)$.  Let $f(x),g(x) \in R[x]$ such that for almost all $k\in R$ we have that $g(k)\vert f(k)$. Let $A=\{k_1,\ldots,k_n\}$ be a finite subset of $R$ such that $g(k)\vert f(k)$ for all $k \in R-A$. Let $k_1,\ldots,k_s \in A$ such that $g(k_i)\not= 0$ for $i=1,\ldots,s$ and let $\beta = g(k_1)\cdots g(k_s)$. If $s=0$, let $\beta = 1$. Then for all $k\in R$ such that $g(k)\not= 0$ we have $g(k)\vert \beta f(k)$.  We can write $g(x)=\alpha h(x)$ where $h(x)$ is primitive with $\deg h=\deg g \geq 1$ and $\alpha$ is the content of $g(x)$.
Let $k \in R$ such that $h(k)\not=0$. Therefore $g(k)\not=0$ and $g(k)\vert \beta f(k)$; but $h(k)\vert g(k)$, so $h(k)\vert\beta f(k)$. By hypothesis, $h(x)\vert\beta f(x)$ in $R[x]$. Hence, there exists $p(x)\in R[x]$ such that $\beta f(x)=p(x)h(x)$ and so $\alpha\beta f(x)= p(x)(\alpha h(x)) = p(x)g(x)$. Therefore $f(x) = {(\alpha\beta)}^{-1}p(x)g(x)$ where ${(\alpha\beta)}^{-1}p(x) \in K[x]$, i.e. $\frac{f(x)}{g(x)}\in K[x]$. Hence, $R$ is a D-ring.
\end{proof}
The following result provides a  positive answer to our original question about divisibility of polynomials with integer coefficients.
\begin{theorem}
If given $f(x),g(x) \in \Z[x]$ with $g(x)$ primitive, $\deg g(x)\geq 1$ and
for all $k \in \Z$, $(g(k)\not=0 \Rightarrow g(k)\vert f(k))$, then $g(x)\vert f(x)$ in $\Z[x]$.
\end{theorem}
\begin{proof}
Since $\Z$ is a unique factorization domain, the result follows from Proposition \ref{dringepp} together with Lemma \ref{Zisadring}.
\end{proof}
\begin{example}
For all $n\geq 0$ consider $p_n(x)$ and $q_n(x)$ defined as follows.
\begin{align*}
p_0(x)&=1,& p_1(x)&=x,& p_{n+1}(x)&=2xp_n(x)-p_{n-1}(x),\\
q_0(x)&=0,& q_1(x)&=1,& q_{n+1}(x)&=2xq_n(x)-q_{n-1}(x).\\
\end{align*}
\begin{table}[ht]\label{table}
\centering
\begin{tabular}{|l|l|l|}\hline
$n$ & $p_{n}(x)$ & $q_{n}(x)$\\
\hline $0$ & $1$ & $0$ \\
\hline $1$ & $x$ & $1$ \\
\hline $2$ & $2x^2-1$ & $2x$ \\
\hline $3$ & $4x^3-3x$ & $4x^2 - 1$ \\
\hline $4$ & $8x^4 - 8x^2 + 1$ & $8x^3 -4x$ \\
\hline $5$ & $16x^5 -20 x^3 +5x$ & $16x^4 -12x^2 +1$ \\
\hline $6$ & $32x^6-48x^4+18x^2-1$ & $32x^5-32x^3+6x$ \\
\hline $7$ & $64x^7-112x^5+56x^3-7x$ & $64x^6 -80x^4+24x^2-1$ \\
\hline $8$ & $128x^8-256x^6+160x^4-32x^2+1$ & $128x^7-192x^5+80x^3-8x$\\
\hline
\end{tabular}
\caption{Polynomials $p_n(x)$ and $q_n(x)$ for $n=0,\ldots, 8$}
\end{table}

In \cite{matiya} is proved the following congruence for all $a\in \Z-\{0,-1\}$ and $n\geq 1$.
\begin{equation*}
q_{2n}(a)\equiv 0\mod p_n(a),
\end{equation*}
This implies that for all $a\in \Z$ with $p_n(a)\not=0$ then $p_n(a)|q_{2n}(a)$. Can we deduce $p_n(x)|q_{2n}(x)$ as polynomials? We can check this for example when $n=4$: note that $q_8(x)=128x^7-192x^5+80x^3-8x=8 x (2 x^2-1) (8 x^4-8x^2+1)$ and $p_4(x)=8x^4 - 8x^2 + 1$ which clearly shows $p_4(x)|q_8(x)$. It is straightforward to prove that for all $n\geq 1$ the polynomials $p_n(x)$ are primitive with $\deg p_n \geq 1$.  It follows from Proposition \ref{dringepp} that for all $n\geq 1$, $p_n(x)|q_{2n}(x)$. 
\end{example}

\end{document}